\documentclass[12pt]{article}
\pdfoutput=1

\usepackage{amsmath}
\usepackage[utf8]{inputenc}
\usepackage[UKenglish]{babel}

\usepackage[normalem]{ulem} 
\usepackage{aliascnt}
\usepackage{amssymb}
\usepackage{amsthm}
\usepackage{faktor}
\usepackage{mathtools}
\usepackage{amsfonts}
\usepackage{enumerate}
\usepackage{enumitem}
\usepackage{multicol}
\usepackage{float}
\usepackage{tikz}
\usepackage{tikz-cd} 
\usepackage{comment}
\usepackage[round,comma,sort,
]{natbib}
\usepackage{hyperref}
\hypersetup{
    colorlinks=true,       
    linkcolor=blue,          
    citecolor=olive,       
    filecolor=magenta,      
    urlcolor=blue          
}

\usepackage[nameinlink]{cleveref}

\newcommand{\Z}{\mathbb{Z}}

\newcommand{\gen}[1]{\langle #1 \rangle}

\newcommand{\gp}[2]{\langle #1 \, | \, #2 \rangle}
\newcommand{\sgp}[1]{\langle #1 \rangle}

\def\coloneqq{\mathrel{\mathop\mathchar"303A}\mkern-1.2mu=}

\begin{document}

\title{The submonoid and rational subset membership problems for Artin groups}
\author{Islam Foniqi}

\date{}
\maketitle

\theoremstyle{plain}
\newtheorem{theorem}{Theorem}[section]
\newtheorem{prop}[theorem]{Proposition}
\newtheorem{conj}[theorem]{Conjecture}
\newtheorem{notation}[theorem]{Notation}
\newtheorem*{note}{Note}
\theoremstyle{definition}

\newtheorem{question}[theorem]{Question}

\setlength{\parindent}{0em}
\setlength{\parskip}{0.5em} 
\author{}

\newaliascnt{conjecture}{theorem}
\newtheorem{conjecture}[conjecture]{Conjecture}
\aliascntresetthe{conjecture}
\providecommand*{\conjectureautorefname}{Conjecture}

\newaliascnt{lemma}{theorem}
\newtheorem{lemma}[lemma]{Lemma}
\aliascntresetthe{lemma}
\providecommand*{\lemmaautorefname}{Lemma}

\newaliascnt{cor}{theorem}
\newtheorem{cor}[cor]{Corollary}
\aliascntresetthe{cor}
\providecommand*{\corautorefname}{Corollary}

\newaliascnt{claim}{theorem}
\newtheorem{claim}[claim]{Claim}

\newaliascnt{notation}{theorem}
\aliascntresetthe{notation}
\providecommand*{\notationautorefname}{Notation}

\aliascntresetthe{claim}
\providecommand*{\claimautorefname}{Claim}

\newaliascnt{remark}{theorem}
\newtheorem{remark}[remark]{Remark}
\aliascntresetthe{remark}
\providecommand*{\remarkautorefname}{Remark}

\newtheorem*{claim*}{Claim}
\theoremstyle{definition}

\newaliascnt{definition}{theorem}
\newtheorem{definition}[definition]{Definition}
\aliascntresetthe{definition}
\providecommand*{\definitionautorefname}{Definition}

\newaliascnt{example}{theorem}
\newtheorem{example}[example]{Example}
\aliascntresetthe{example}
\providecommand*{\exampleautorefname}{Example}

\newaliascnt{question}{theorem}
\aliascntresetthe{question}
\providecommand*{\exampleautorefname}{Question}

\begin{abstract}
We demonstrate that the submonoid membership problem and the rational subset membership problem are equivalent in Artin groups. Both these problem are undecidable in a given Artin group if and only if the group embeds the right-angled Artin groups of rank~$4$ over a path or a square; and this can be characterized using only the defining graph of the Artin group. These results generalize the ones by Lohrey - Steinberg for right-angled Artin groups.
Moreover,
both these decision problems are decidable for a given Artin group if and only if the group is subgroup separable.
This equivalence for right-angled Artin groups is provided by Lohrey - Steinberg and Metaftsis - Raptis. The equivalence for general Artin groups comes from some observations here and the characterization of separable Artin groups by Almeida - Lima.
\end{abstract}

\renewcommand{\thefootnote}{\fnsymbol{footnote}} 

\thefootnote{
}


\section{Introduction and preliminaries}

In this article we investigate two decision problems, namely the submonoid, and the rational subset membership problem, in the well-known class of Artin groups.

\subsection{Artin groups}

Artin groups are ﬁnitely presented groups deﬁned by a ﬁnite simplicial labeled graph. They form a huge class of groups, which includes free groups, free abelian groups, the classical braid groups, right-angled Artin groups, etc. There are not many results that hold for all Artin groups, e.g. it is conjectured and still an open problem, that all Artin Groups have solvable word problem.

A finite simplicial graph~$\Gamma$ is a tuple~$\Gamma = (V, E)$, where~$V$ is a finite set whose elements are called {\it vertices},~$E$ is a set of two-element subsets of~$V$ whose elements are called {\it edges}. When referring to a graph~$\Gamma$, we will  denote by~$V\Gamma$,~$E\Gamma$ its vertex set, and edge set respectively.

An {\it  Artin graph} consists of a finite simplicial graph~$\Gamma = (V, E)$ together with a function~$m\colon E\to \{2,3,4,\dots \}$ called {\it labeling} of the edges.

Given an Artin graph~$\Gamma$, the corresponding {\it Artin group based on~$\Gamma$} (also known as the {\it Artin--Tits group}) and denoted by~$A(\Gamma)$ is the group with presentation
$$A(\Gamma) \coloneqq \langle\, V \mid (u,v)_{m(\{u,v\})}= (v,u)_{m(\{u,v\})}  \text{ for all } \{u,v\}\in E\,\rangle,$$
where~$(u,v)_n$ denotes the prefix  of length~$n$ of the infinite alternating word~$uvuvuv\dots$. 

Associated to an Artin graph, we can also construct the {\it Coxeter group based on~$\Gamma$} which is the group with presentation 
$$C(\Gamma) \coloneqq \langle\, V \mid v^2=1  \text{ for all }  v\in V, \, (u,v)_{m(\{u,v\})}= (v,u)_{m(\{u,v\})}   \text{ for all }  \{u,v\}\in E\,\rangle.$$

An Artin graph~$\Gamma$ and the corresponding group~$A(\Gamma)$ are called of {\it spherical type} if the associated Coxeter group~$C(\Gamma)$ is finite. 

For~$S \subseteq V$, we denote by~$A_S$ to the subgroup of~$A(\Gamma)$ generated by the vertices of~$S$. 
Subgroups of this form are called {\it standard parabolic subgroups}, and a theorem of \cite{van1983homotopy} shows that~$A_S\cong A(\Delta)$ where~$\Delta$ is the  Artin subgraph of~$\Gamma$ induced by~$S$.


As a convention, whenever an edge~$e = \{a, b\}$ has the label~$2$, we will drop the label, while keeping in mind that the  relation coming from that edge is~$ab = ba$. We will sometimes refer to them as \emph{unlabeled edges}. This will help to present easier some discussions about graphs and right-angled Artin groups, defined in \autoref{Right-angled Artin groups}. 

\subsection{Monoids, groups, and decision problems}
A \emph{monoid} is a set equipped with an associative binary operation and an identity element. They generalize groups, because a monoid where every element has an inverse is a group.

One of the most basic examples of monoids, is the set of non-negative integers with addition, the identity element being~$0$. 

For a non-empty alphabet~$A$ we denote by~$A^{\ast}$ the \emph{free monoid} of all words over~$A$,  including the empty word denoted by~$\varepsilon$.    





Given a subset~$X$ of a monoid, we use~${X}^{\ast}$ to denote the submonoid generated by~$X$, and similarly if~$X$ is a subset of a group then~$\sgp{X}$ denotes the subgroup generated by~$X$.

\subsection{Decision problems}
A \emph{decision problem} is a YES -- or -- NO question on a countable infinite set of inputs. 
A decision problem is called \emph{decidable} if there is an algorithm which takes an input, terminates after a finite amount of time, and correctly answers the question by YES or NO. If such an algorithm does not exist, we refer to the decision problem as \emph{undecidable}.

Word problem is one of the fundamental algorithmic questions in algebra, and in group theory in particular. For a group~$G = \gp{A}{R}$, the word problem asks whether~$w \in (A \cup A^{-1})^{\ast}$ is equal to~$1_G$ in~$G$. This problem is still open for Artin groups.

A more general concept is the \emph{generalized word problem} (also known as the \emph{subgroup membership problem}) for a group~$G$: existence of an algorithm that decides for every element~$g \in G$ and every ﬁnitely generated subgroup~$H \leqslant G$ whether~$g$ belongs to~$H$ or not.

Natural generalizations of these classical decision problems are the submonoid membership problem, and the rational subset membership problem, which we define shortly below.


\subsection{Submonoid and rational subset membership problems}

Let~$M$ be a monoid finitely generated by a set~$A$, and~$\phi:A^{\ast} \rightarrow M$ the corresponding canonical homomorphism. 

The \emph{submonoid membership problem} for~$M$ is the decision problem: 
\begin{itemize} 
\item{\textsc{Input}:} A finite set of words~$S \subseteq A^{\ast}$ and a word~$w \in A^{\ast}$   
\item{\textsc{Question}:}~$\phi(w) \in \phi(S^{\ast})$\,?   
  \end{itemize}
Note that~$\phi(S^{\ast})$ is equal to the submonoid of~$M$ generated by~$\phi(S)$.     
The decidability of this problem is independent of the choice of finite generating set of the monoid. \\
If~$G$ is a group with finite generating set~$A$, the set~$A \cup A^{-1}$ is a finite monoid generating set for~$G$, hence one can define the submonoid membership problem to groups as well.

The set of \emph{rational subsets} of a monoid~$M$ is the smallest subset of the power set of~$M$ which contains 
all finite subsets of~$M$, and is closed under union, product, and Kleene hull. \\
The \emph{Kleene hull} of a subset~$L$ of a monoid~$M$ is just the submonoid~$L^{\ast}$ of~$M$ generated by~$L$. \\
Note that every finitely generated submonoid of~$M$ is a rational subset (being equal to the Kleene hull of a finite set).
A subset~$L \subseteq M$ is rational if and only if~$L = \phi(R)$ for some rational subset~$R$ of~$A^{\ast}$. 


The \emph{rational subset membership problem} for~$M$ is the decision problem: 
\begin{itemize} 
\item{\textsc{Input}:} A rational subset~$R \subseteq M$, and a word~$w \in A^{\ast}$
\item{\textsc{Question}:}~$\phi(w) \in R$\,?   
\end{itemize}
The rational subset membership problem also applies to groups~$G = \gp{A}{R}$ where we view the group~$G$ as a monoid generated by~$A \cup A^{-1}$.
\begin{remark}
Decidability of the rational subset membership problem implies decidability of the submonoid membership problem.
\end{remark}
The rational subset membership problem is more general than the submonoid membership problem. The result~\citep[Corollary~2.8]{bodart2024membership} provides a group with decidable submonoid membership problem, and undecidable rational subset membership problem.

Non-uniform analogues of the problems above consider a fixed finitely generated submonoid (or a rational subset) and ask whether there is an algorithm deciding membership there. \\
In full generality, for any subset~$S$ of~$M$ by the \emph{membership problem for~$S$ within~$M$} we mean the decision problem:
\begin{itemize} 
\item{\textsc{Input}:} 
A word~$w \in A^{\ast}$   
\item{\textsc{Question}:}~$\phi(w) \in S$? 
  \end{itemize}
One can talk as well about the membership problem for~$S$ within a finitely generated group~$G$. 

If~$M$ and~$T$ are finitely generated monoids with~$T \leqslant M$ then: ~$M$ having decidable submonoid (respectively rational subset) membership problem, implies the same for~$T$.

\subsection{Right-angled Artin groups}\label{Right-angled Artin groups}
The class of right-angled Artin groups (shortly RAAGs) is a subclass of Artin groups, when the labeling map~$m$ for the Artin graph~$\Gamma = (V, E)$ satisfies~$m(E) \subseteq \{2\}$; in other words, commutations are the only possible relations between the generators.
The class of RAAGs plays an important role in geometric group theory. For an expanded background on RAAGs we refer the reader to the survey of \cite{charney2007introduction}.

When working with RAAGs, we drop the labels on edges, and we use only the graph structure to define and work with them.
For a finite simplicial graph~$\Gamma = (V, E)$ we use~$A(\Gamma)$ to denote the right-angled Artin group based on~$\Gamma$, defined by the presentation 
$$A(\Gamma) \coloneqq \langle\, V \mid uv = vu \text{ whenever } \{u,v\}\in E\,\rangle.$$

Two important families of examples of RAAGs, include the ones definied by paths and cycles.

\begin{definition}
Let~$P_n$ and~$C_n$ denote the path and cycle graphs on~$n$ vertices. Below we provide~$P_4$ and~$C_4$.

\begin{figure}[H]
\centering
\begin{tikzpicture}[>={Straight Barb[length=7pt,width=6pt]},thick]
\draw[] (-6.5, 0) node[left] {$P_4 =~$};
\draw[fill=black] (-6,0) circle (1pt) node[above] {$a$};
\draw[fill=black] (-5,0) circle (1pt) node[above] {$b$};
\draw[fill=black] (-4,0) circle (1pt) node[above] {$c$};
\draw[fill=black] (-3,0) circle (1pt) node[above] {$d$};
\draw[] (-3, 0) node[right] {$\,\,\, ,$};

\draw[thick] (-6,0) -- (-5,0);
\draw[thick] (-5,0) -- (-4,0);
\draw[thick] (-4,0) -- (-3,0);

\draw[] (1, 0) node[left] {$C_4 =~$};
\draw[fill=black] (2,1) circle (1pt) node[left] {$a$};
\draw[fill=black] (4,1) circle (1pt) node[right] {$b$};
\draw[fill=black] (4,-1) circle (1pt) node[right] {$c$};
\draw[fill=black] (2,-1) circle (1pt) node[left] {$d$};

\draw[thick] (2, 1) -- (4, 1);
\draw[thick] (4,1) -- (4,-1);
\draw[thick] (4,-1) -- (2,-1);
\draw[thick] (2,-1) -- (2,1);
\end{tikzpicture}
\end{figure} 

\end{definition}

In \citep{lohrey2008submonoid} it is proved that a RAAG~$A(\Gamma)$ has decidable submonoid membership problem if and only if it has decidable rational subset membership problem if and only if~$\Gamma$ does not contain both~$P_4$ and~$C_4$, as induced subgraphs. A complete characterization of RAAGs with decidable subgroup membership problem is not known; it is for example unknown whether~$A(C_5)$ has decidable subgroup membership problem.

\subsection{Subgroup separability}

A group~$G$ is called \emph{subgroup separable}
if every ﬁnitely generated subgroup of~$G$ is equal to an intersection of subgroups of ﬁnite index of~$G$. For ﬁnitely presented groups, being subgroup separable implies the solvability of the generalized word problem.

In \citep{metaftsis2008profinite} there is a criterion for the subgroup separability for RAAGs, where it is shown that a RAAG~$A(\Gamma)$ is subgroup separable if and only if~$\Gamma$ does not contain both the path~$P_4$ and the square~$C_4$, as an induced subgraphs. \\
The characterization of subgroup separable braid groups was solved by \cite*{dasbach2001automorphism}. The classification of subgroup separable Artin groups is done by \cite{almeida2021subgroup}; and recently in \citep{almeida2024subgroupseparabilityartingroups} the same authors provide another version of the characterization that uses only the defining graph.

\subsection{Presentation of results and strategy}
Often one would like to characterize a certain class of groups~$\mathcal{C}$ that satisfies a certain property~$\mathcal{P}$. 
In our case the class of groups will be the one of Artin groups, and the properties we will discuss will be submonoid (and rational subset) membership problems, and some other ones closely related to these two.

Some of the properties we study have some nice closure features, which will be helpful to provide characterizations of the given properties; some of the nice features are given below.
\begin{theorem}[{\citealp[Theorem 1]{lohrey2008submonoid}}]\label{thm: class of groups with decidable rational subset membership}
Let~$\mathcal{C}$ be the smallest class of groups such that:
\begin{enumerate}
    \item The trivial group~$1$ is in~$\mathcal{C}$,
    \item If~$G \in \mathcal{C}$, then~$H \times \Z \in \mathcal{C}$,
    \item If~$G \in \mathcal{C}$ and~$H \leqslant G$ is finitely generated, then~$H \in \mathcal{C}$,
    \item If~$G \in \mathcal{C}$ and~$G \leqslant H$ of finite index, then~$H \in \mathcal{C}$,   
    \item Fundamental groups of graphs of groups with finite edge groups and vertex groups belonging to~$\mathcal{C}$ are in~$\mathcal{C}$. In particular, the class~$\mathcal{C}$ is closed under free products.
\end{enumerate}
Then, for every group~$G\in \mathcal{C}$, the rational subset membership problem is decidable.
\end{theorem}
\begin{remark}
Submonoids are particular examples of rational subsets, hence groups belonging to the class~$\mathcal{C}$ referred in \autoref{thm: class of groups with decidable rational subset membership} have decidable submonoid membership problem.
\end{remark}
\begin{notation}
In this article we will use~$\mathcal{A}$ to denote the class of Artin groups. Now we define the following subclasses of~$\mathcal{A}$ to express our results.
\begin{itemize}
\item~$\mathcal{A}(Rat)$ denotes Artin groups with decidable rational subset membership problem.
\item~$\mathcal{A}(Mon)$ denotes Artin groups with decidable submonoid membership problem.
\item~$\mathcal{A}(SgSep)$ denotes Artin groups that are subgroup separable.
\item~$\mathcal{A}(P_4 \lor C_4)$ denotes Artin groups
that contain~$A(P_4)$ or~$A(C_4)$ as subgroups.
\end{itemize}
\end{notation}

\begin{remark}
If one denotes the class of RAAGs by~$\mathcal{RA}$, then one can make the analogous notation as above for subclasses of RAAGs.
\end{remark}

In \autoref{Review of results about RAAGs}, we provide a literature review of results about RAAGs that show:
\[
\mathcal{RA}(Rat) = \mathcal{RA}(Mon) = \mathcal{RA}(SgSep) = \mathcal{RA} \setminus \mathcal{RA}(P_4 \lor C_4).
\]
Moreover, a RAAG~$A(\Gamma)$ belongs to one of these classes if and only if~$\Gamma$ does not contain~$P_4$ or~$C_4$ as induced subgraphs.


In \autoref{Going from RAAGs to all Artin groups}, we use the equivalence of these decision problems in RAAGs and several observations to demonstrate:
\[
\mathcal{A}(Rat) = \mathcal{A}(Mon) = \mathcal{A}(SgSep) = \mathcal{A} \setminus \mathcal{A}(P_4 \lor C_4).
\]

Clearly one has~$\mathcal{A}(Rat) \subseteq \mathcal{A}(Mon)$ because submonoids are in particular rational subsets.

The containment~$\mathcal{A}(Mon) \subseteq \mathcal{A} \setminus \mathcal{A}(P_4 \lor C_4)$ follows from the fact that both RAAGs~$A(P_4)$ and~$A(C_4)$ have undecidable submonoid membership problem. 

The equivalence~$\mathcal{A}(SgSep) = \mathcal{A} \setminus \mathcal{A}(P_4 \lor C_4)$ is proved by Almeida -- Lima, in two papers.

Our main discussions include the characterization of~$\mathcal{A}(P_4 \lor C_4)$ and its complement. We also discuss thoroughly the containment
\[
\mathcal{A} \setminus \mathcal{A}(P_4 \lor C_4) \subseteq \mathcal{A}(Rat).
\]
We provide a proof with a slightly different flavour (see \autoref{thm: when we have decidable RSMP}) to the equivalent result, provided by Almeida -- Lima, about the containment~$\mathcal{A} \setminus \mathcal{A}(P_4 \lor C_4) \subseteq \mathcal{A}(SgSep)$.

\section{Review of results about RAAGs}\label{Review of results about RAAGs}
This section is only about RAAGs, some of their properties, and decision problems in them. The main result appearing in this section is the following:
\begin{theorem}\label{thm: main theorem for RAAGs}
One has~$\mathcal{RA}(Rat) = \mathcal{RA}(Mon) = \mathcal{RA}(SgSep) = \mathcal{RA} \setminus \mathcal{RA}(P_4 \lor C_4)$. Moreover, a RAAG~$A(\Gamma)$ belongs to one of these classes if and only if~$\Gamma$ does not contain both~$P_4$ and~$C_4$ as induced subgraphs.
\end{theorem} 
We will provide a summary of the theorem above, by gathering the appropriate results from literature about RAAGs, regarding each equality.
\begin{theorem}[\citealp{mikhailova1966occurrence}]
The RAAG~$A(C_4) = F_2 \times F_2$ contains a subgroup where membership is undecidable.
\end{theorem}
As subgroups are particular examples of rational subsets, one has that~$A(C_4)$ contains a rational subset where membership is undecidable.

On the other hand~$A(P_4)$ has decidable subgroup membership problem, see \citep*[Corollary 1.3]{kapovich2005foldings}, as~$P_4$ is a chordal graph. However, one has other undecidable results for~$A(P_4)$.
\begin{theorem}[{\citealp[Theorem 5]{lohrey2008submonoid}}]
The RAAG~$A(P_4)$ contains a submonoid where membership is undecidable.
\end{theorem}
In particular,~$A(P_4)$ contains a rational subset where membership is undecidable.

As implications of the results above, one has:
\[
\mathcal{RA}(Rat), \, \mathcal{RA}(Mon) \subseteq \mathcal{RA} \setminus \mathcal{RA}(P_4 \lor C_4).
\]
Let~$A(\Gamma)$ be a RAAG based on the simplicial graph~$\Gamma$. The following two results demonstrate that~$A(\Gamma)$ can have~$A(P_4)$, respectively~$A(C_4)$, as a subgroup if and only if one sees~$P_4$, respectively~$C_4$, as an induced subgraph of~$\Gamma$. One direction of this result is clear, since:
\begin{lemma}\label{lem: injective graphs injective raags}
If~$\Delta$ is an induced subgraph of the simplicial graph~$\Gamma$ then the map on RAAGs
\[
i \colon A(\Delta) \longrightarrow A(\Gamma), \text{ given by } u \mapsto u \text{ for all } u \in V\Delta,
\]
defines an injective morphism. 
\end{lemma}
\begin{proof}
This can be seen by considering the retraction morphism:
\[
\rho \colon A(\Gamma) \longrightarrow A(\Delta), \text{ induced by } u \mapsto u \text{ for } u \in V\Delta, \text{ and } v \mapsto 1 \text{ for } v \in V\Gamma \setminus V\Delta,
\]
which satisfies~$\rho \circ i = id_{A(\Delta)}$; showing that~$i$ is indeed injective. 
\end{proof}
As a converse to the statement above, one has the following two results.
\begin{theorem}[{\citealp[Corollary 3.8]{kambites2009commuting}}]
The RAAG~$A(\Gamma)$ contains a subgroup isomorphic to~$A(C_4)$ if and only if~$\Gamma$ contains an induced square~$C_4$.
\end{theorem}
\begin{theorem}[{\citealp[Theorem 1.7]{kim2013embedability}}]
If~$A(P_4) \xhookrightarrow{} A(\Gamma)$ is a group embedding then~$\Gamma$ contains an induced path~$P_4$.
\end{theorem}

\begin{definition}\label{def: elementary RAAGs}
Simplicial graphs not containing~$P_4$ and~$C_4$ are called \emph{transitive forests}. Right-angled Artin groups based on transitive forests are called \emph{elementary RAAGs}.
\end{definition}

The result \citep[Lemma 2]{lohrey2008submonoid} shows that the class of elementary RAAGs consists precisely of the subclass~$\mathcal{RE}$ of RAAGs with the following properties:
\begin{itemize}
\item[(i)] The trivial group~$1$ is in~$\mathcal{RE}$,
\item[(ii)] If~$E \in \mathcal{E}$, then~$E \times \Z \in \mathcal{RE}$,
\item[(iii)] If~$E_1, E_2 \in \mathcal{E}$ then~$E_1 \ast E_2 \in \mathcal{RE}$, i.e.~$\mathcal{E}$ is closed under free products.
\end{itemize}

\begin{remark}
Note that~$\mathcal{RE}$ is a subclass of the class of groups~$\mathcal{C}$ referred to in \autoref{thm: class of groups with decidable rational subset membership}. Hence one has that the class of elementary RAAGs has decidable rational subset membership problem. Furthermore, since submonoids are rational subsets, we obtain:
\[
\mathcal{RA} \setminus \mathcal{RA}(P_4 \lor C_4) \subseteq \mathcal{RA}(Rat) \subseteq  \mathcal{RA}(Mon).
\]
Since~$\mathcal{RA}(Mon) \subseteq \mathcal{RA}(P_4 \lor C_4)$, one has a full equality:
\[
\mathcal{RA}(Rat) = \mathcal{RA}(Mon) = \mathcal{RA} \setminus \mathcal{RA}(P_4 \lor C_4).
\]
\end{remark}
Now the equality of the above three with~$\mathcal{RA}(SgSep)$ follows from the characterization of RAAGs that are subgroup separable, provided by the following:
\begin{theorem}[{\citealp[Theorem 2]{metaftsis2008profinite}}]
A RAAG~$A(\Gamma)$ is subgroup separable if and only if~$\Gamma$ does not contain~$P_4$ or~$C_4$ as induced subgraphs, i.e. if and only if~$A(\Gamma)$ does not contain~$A(P_4)$ or~$A(C_4)$ as subgroups.
\end{theorem}

\begin{remark}
There are other properties of RAAGs that have this exactly same characterization. For example, if one wants to classify all RAAGs whose finitely generated subgroups are again RAAGs, then we obtain exactly the elementary ones, see \citep{droms1987subgroups}. So one can add many other properties in the equality of \autoref{thm: main theorem for RAAGs}
\end{remark}

Some authors call both graphs~$P_4$ and~$C_4$ \emph{poisonous}; similarly~$A(P_4)$ and~$A(C_4)$ get called \emph{poisonous groups}. They tend to be the smallese examples to not satisfy a certain property. 

\begin{remark}
Note that none of the groups~$A(P_4)$ and~$A(C_4)$ is a subgroup of the other, because none of~$P_4$ and~$C_4$ is an induced subgraph of the other.
\end{remark}

\section{Going from RAAGs to all Artin groups}\label{Going from RAAGs to all Artin groups}

When extending some of the previous results from RAAGs to another class of groups~$\mathcal{C}$ that contains RAAGs, one searches for the 'poisonous pieces'~$A(P_4)$ and~$A(C_4)$ in the larger class. Then any property that satisfies the subgroup closure and that fails for~$A(P_4)$ and~$A(C_4)$, fails for~$G \in  \mathcal{C}$ as well whenever~$G$ contains~$A(P_4)$ or~$A(C_4)$.

This will be the theme of this section for the class of all Artin groups~$\mathcal{A}$. One can hope to generalize the results from RAAGs to all Artin groups by tackling the following question.
\begin{question}\label{question: about containing P_4 and C_4}
Which Artin groups contain~$A(P_4)$ or~$A(C_4)$ as subgroups?
\end{question}
This is precisely the question that \cite{almeida2021subgroup} answer to classify subgroup separable Artin groups: an Artin group is subgroup separable if and only if it can be obtained from~$\Z$ and dihedral Artin groups (see \autoref{subsec: Dihedral Artin groups}), via a ﬁnite sequence of free products, and direct products with~$\Z$. This generalizes the Metaftsis-Raptis criterion for RAAGs.

\begin{theorem}[{\citealp[Theorem 1.1]{almeida2024subgroupseparabilityartingroups}}]\label{thm: separable subgroup Artin groups}
An Artin group~$A(\Gamma)$ is subgroup separable if and only if~$\Gamma$ does not contain the following induced subgraphs:
\begin{itemize}
\item[(1)] The usual poisonous graphs:
\begin{figure}[H]
\centering
\begin{tikzpicture}[scale = 0.8, >={Straight Barb[length=7pt,width=6pt]},thick]
\draw[] (-6.5, 0) node[left] {$P_4 =~$};
\draw[fill=black] (-6,0) circle (1pt) node[above] {$a$};
\draw[fill=black] (-5,0) circle (1pt) node[above] {$b$};
\draw[fill=black] (-4,0) circle (1pt) node[above] {$c$};
\draw[fill=black] (-3,0) circle (1pt) node[above] {$d$};
\draw[] (-3, 0) node[right] {$\,\,\, ,$};

\draw[thick] (-6,0) -- (-5,0);
\draw[thick] (-5,0) -- (-4,0);
\draw[thick] (-4,0) -- (-3,0);

\draw[] (1, 0) node[left] {$C_4 =~$};
\draw[fill=black] (2,1) circle (1pt) node[left] {$a$};
\draw[fill=black] (4,1) circle (1pt) node[right] {$b$};
\draw[fill=black] (4,-1) circle (1pt) node[right] {$c$};
\draw[fill=black] (2,-1) circle (1pt) node[left] {$d$};

\draw[thick] (2, 1) -- (4, 1);
\draw[thick] (4,1) -- (4,-1);
\draw[thick] (4,-1) -- (2,-1);
\draw[thick] (2,-1) -- (2,1);
\end{tikzpicture}
\end{figure} 

\item[(2)] Chordal Artin squares:
\begin{figure}[H]
\centering
\begin{tikzpicture}[scale = 0.8, >={Straight Barb[length=7pt,width=6pt]},thick]

\draw[fill=black] (2,1) circle (1pt) node[left] {$a$};
\draw[fill=black] (4,1) circle (1pt) node[right] {$b$};
\draw[fill=black] (4,-1) circle (1pt) node[right] {$c$};
\draw[fill=black] (2,-1) circle (1pt) node[left] {$d$};

\draw[] (1.5, 0) node[left] {$S_p =~$};

\draw[thick] (2, 1) -- (4, 1);
\draw[thick] (4,1) -- (4,-1);
\draw[thick] (4,-1) -- (2,-1);
\draw[thick] (2,-1) -- (2,1);
\draw[thick] (2,-1) --  (4,1);
\draw[] (4.5, 0) node[right] {\text{, and}};
\draw[fill=black] (3, -0.12) circle (0pt) node[above] {$p$};


\draw[fill=black] (9,1) circle (1pt) node[left] {$a$};
\draw[fill=black] (11,1) circle (1pt) node[right] {$b$};
\draw[fill=black] (11,-1) circle (1pt) node[right] {$c$};
\draw[fill=black] (9,-1) circle (1pt) node[left] {$d$};

\draw[] (8.5, 0) node[left] {$S_{p, q} =~$};

\draw[thick] (9, 1) -- (11, 1);
\draw[thick] (11,1) -- (11,-1);
\draw[thick] (11,-1) -- (9,-1);
\draw[thick] (9,-1) -- (9,1);
\draw[thick] (9,-1) --  (11,1);
\draw[thick] (11,-1) --  (9,1);
\draw[fill=black] (10.25, 0.1) circle (0pt) node[above] {$p$};
\draw[fill=black] (10.2, -0.15) circle (0pt) node[below] {$q$};
\end{tikzpicture}
\end{figure} 
with~$p, q > 2$.
\item[(3)] The following~$3$-vertex connected graphs:
\begin{itemize}
\item[(i)] The~$3$-vertex path~$P_3^{m, n}$ with~$m + n > 4$,
\begin{figure}[H]
\centering
\begin{tikzpicture}[>={Straight Barb[length=7pt,width=6pt]},thick]
\draw[fill=black] (0,0) circle (1.5pt) node[below] {$a$};
\draw[fill=black] (2,0) circle (1.5pt) node[below] {$b$};
\draw[fill=black] (4,0) circle (1.5pt) node[below] {$c$};
\draw[thick] (0,0) -- (2,0);
\draw[thick] (2,0) -- (4,0);
\draw[fill=black] (1,0)  node[above] {$m$};
\draw[fill=black] (3,0)  node[above] {$n$};
\end{tikzpicture}
\end{figure} 
\item[(ii)] and the triangles~$\Delta_{l, m, n}$ with at most one edge with label~$2$:
\begin{figure}[H]
\centering
\begin{tikzpicture}[>={Straight Barb[length=7pt,width=6pt]},thick]
\draw[] (-1, 0.8) node[left] {$\Delta_{l,m,n} =~$};
\draw[fill=black] (0,0) circle (1.5pt) node[below] {$a$};
\draw[fill=black] (2,0) circle (1.5pt) node[below] {$c$};
\draw[fill=black] (1,1.7) circle (1.5pt) node[above] {$b$};
\draw[thick] (0,0) -- node[below]{$n$}  (2,0);
\draw[thick] (2,0) -- node[right]{$m$}  (1,1.7);
\draw[thick] (0,0) -- node[left]{$l$} (1,1.7);
\end{tikzpicture}
\end{figure} 
\end{itemize}
\end{itemize}
\end{theorem}

\begin{remark}\label{rem: other poisenous graphs contain P_4, c_4}
We will see in \autoref{ex: chordal Artin squares} that~$A(C_4) \xhookrightarrow{} A(\Gamma)$ if~$\Gamma$ is any of the chordal Artin graphs in category~(2). On the other hand~$A(P_4) \xhookrightarrow{} A(\Gamma)$ if~$\Gamma$ is any of the~$3$-vertex Artin graphs in category~(3) by \autoref{lem: P_3^{m,n} contains P_4} and \autoref{cor: Delta_{l, m, n} contains P_4}.
\end{remark}
The main result appearing in this section is the following:
\begin{theorem}\label{thm: main theorem}
One has~$\mathcal{A}(Rat) = \mathcal{A}(Mon) = \mathcal{A}(SgSep) = \mathcal{A} \setminus \mathcal{A}(P_4 \lor C_4)$.
\end{theorem}

\begin{definition}
Let~$\mathcal{A}(\mathcal{E})$ be the smallest subclass of Artin groups such that:
\begin{itemize}
\item~$\mathcal{A}(\mathcal{E})$ contains all Artin groups of ranks at most 2;
\item If~$E \in \mathcal{A}(\mathcal{E})$, then~$\Z \times E \in \mathcal{A}(\mathcal{E})$.
\item If~$E_1, E_2 \in \mathcal{A}(E)$, then the free product~$E_1 \ast E_2 \in \mathcal{A}(\mathcal{E})$;
\end{itemize}
We refer to groups in~$\mathcal{A}(\mathcal{E})$ as \emph{elementary Artin groups}.
\end{definition}

Results \citep[Theorem A, and Corollary A]{almeida2021subgroup} yield:~$
\mathcal{A}(\mathcal{E}) = \mathcal{A}(SgSep)$.

\subsection{Dihedral Artin groups}\label{subsec: Dihedral Artin groups}
Let~$\delta_m$ be the Artin graph:
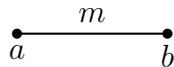
\begin{figure}[H]
\centering
\begin{tikzpicture}[>={Straight Barb[length=7pt,width=6pt]},thick]
\draw[fill=black] (0,0) circle (1.5pt) node[below] {$a$};
\draw[fill=black] (2,0) circle (1.5pt) node[below] {$b$};
\draw[thick] (0,0) -- (2,0);
\draw[fill=black] (1,0)  node[above] {$m$};
\end{tikzpicture}
\end{figure} 
The Artin group~$D_{m} \coloneqq A(\delta_m) = \gp{a, b}{(a, b)_m = (b, a)_m}$, is called the \emph{dihedral Artin group}. 

\begin{definition}
By a \emph{star-shaped graph}~$St_n$ we refer to the graph with~$n+1$ vertices, where one is central and connected to all the other~$n$ vertices, and these are the only edges of the graph. Graphically,~$St_n$ looks like this:
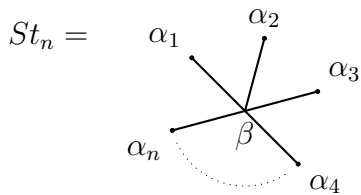
\begin{figure}[H]
\centering
\begin{tikzpicture}[rotate = -15, scale = 0.5,>={Straight Barb[length=7pt,width=4pt]}] 

\draw[fill=black] (0,0) circle (0pt) node[below] {$\beta$};

\draw[fill=black] (-4, 1) circle (0pt) node[left] {$St_n =~$};

\draw[fill=black] ({2*cos(150)}, {2*sin(150)}) circle (1.5pt) node[above left] {$\alpha_1$};

\draw[fill=black] ({2*cos(90)}, {2*sin(90)}) circle (1.5pt) node[above] {$\alpha_2$};

\draw[fill=black] ({2*cos(30)}, {2*sin(30)}) circle (1.5pt) node[above right] {$\alpha_3$};

\draw[fill=black] ({2*cos(330)}, {2*sin(330)}) circle (1.5pt) node[below right] {$\alpha_4$};

\draw[fill=black] ({2*cos(210)}, {2*sin(210)}) circle (1.5pt) node[below left] {$\alpha_{n}$};


\draw [dotted,domain=220:320] plot ({2*cos(\x)}, {2*sin(\x)});

\draw[thick] (0,0) -- ({2*cos(330)}, {2*sin(330)});
\draw[thick] (0,0) -- ({2*cos(30)}, {2*sin(30)});
\draw[thick] (0,0) -- ({2*cos(90)}, {2*sin(90)});
\draw[thick] (0,0) -- ({2*cos(150)}, {2*sin(150)});
\draw[thick] (0,0) -- ({2*cos(210)}, {2*sin(210)});
\end{tikzpicture}
\end{figure}
\end{definition}

\begin{theorem}
Dihedral Artin groups are finite extensions of RAAGs based on star-shaped graphs.
\end{theorem}

For a sketch of the proof of the theorem above we distinguish two cases: (i) the label~$m$ is even, and (ii) the label~$m$ is odd; and we discuss these cases separately below.

\begin{itemize}
\item[(i)] {\bf Even dihedral Artin groups.}

For the even case denote~$m = 2n$. Then~$D_{2n} \cong \gp{a, x}{ax^n = x^n a}$ via the correspondence~$a \longleftrightarrow a$, and~$ab \longleftrightarrow x$. In \cite[Lemma 2.5.]{antolin2024subgroups} we show that the subgroup~$H_n = \gen{x^n, a, xax^{-1}, \ldots, x^{n-1}ax^{-n+ 1}}$ is normal and of finite index~$n$ in~$D_{2n}$. Moreover,~$H_n$ is isomorphic to~$\mathbb{F}_n\times \mathbb{Z}$ (with~$x^n$ being the central element).

The group~$H_n$ is a RAAG, concretely~$H_n = A(St_{n})$, with the correspondence~$\beta = x^n$, and~$\alpha_{i+1} = x^iax^{i-1}$ for~$0 \leq \alpha \leq n-1$.

\item[(ii)] {\bf Odd dihedral Artin groups.} 

In this case~$D_{m} \cong \gp{x, y}{x^m = y^2}$ via the correspondence~$ab \longleftrightarrow x$, and~$(b, a)_m \longleftrightarrow y$; which means that~$D_m$ is actually a torus knot. By~\cite[Theorem 1.5.]{katayama2017raags} torus knots (hence odd dihedral Artin groups as well) are finite extensions of a direct product of~$\Z$ and a free group; for our case~$D_m$ has a subgroup of index~$2m$ which is isomorphic to~$\Z \times F_{m-1}$, which is a RAAG given by~$A(St_{m-1})$. 
\end{itemize}

\begin{cor}\label{cor: dihedral Artin groups are in A(Rat)}
Since~$D_m$ is a finite extension of an elemetary RAAG, \autoref{thm: class of groups with decidable rational subset membership} implies that~$D_m$ has decidable rational subset membership problem, i.e. for all integers~$m \geq 2$, the group~$D_m$ belongs to~$\mathcal{A}(Rat)$.
\end{cor}
The discussion of this section and \autoref{lem: injective graphs injective raags} gives the following:
\begin{remark}\label{rem: dihedral Artin groups with m > 2 contain F_2}
The free group~$F_2$ (of rank~$2$) is a subgroup of the dihedral Artin group~$D_m$ if and only if~$m > 2$.
\end{remark}

\subsection{Finding RAAGs inside Artin groups}
One way to find poisonous RAAGs in certain Artin groups would be the observation coming from \autoref{rem: dihedral Artin groups with m > 2 contain F_2}; which is useful to find~$A(C_4) \simeq F_2 \times F_2$ inside Artin groups.
\begin{example}\label{ex: chordal Artin squares}
Let~$p, q > 2$ be integers, and consider the  chordal Artin squares~$S_p$ and~$S_{p, q}$ from category (2) of \autoref{thm: separable subgroup Artin groups}. Note that~$A(S_p) \simeq F_2 \times D_p$, and~$A(S_{p, q}) \simeq D_p \times D_q$. Since~$p, q > 2$, \autoref{rem: dihedral Artin groups with m > 2 contain F_2} implies~$F_2 \xhookrightarrow{} D_p$ and~$F_2 \xhookrightarrow{} D_q$; hence, one has these two inclusions: 
$A(C_4) \xhookrightarrow{} D_2 \times D_p = A(S_p)$, and~$A(C_4) \xhookrightarrow{} D_p \times D_q = A(S_{p,q})$.
\end{example}
Another way to see this, comes from Tits Conjecture studied by \cite{crisp2001solution}. We give only few details about it below.
\subsubsection*{Tits Conjecture}
In \citep{crisp2001solution} the authors proved the full Tits Conjecture for Artin groups. This conjecture gives presentations for the subroups of Artin groups generated by powers (at least~$2$) of vertices, and it shows that these presentations are RAAG presentations. Below we give just the case of squares of vertices (which was conjectured by Tits); see \cite[Theorem 1.]{crisp2001solution} for the full result.

\begin{theorem}[{\citealp[Corollary 2]{crisp2001solution}}]\label{thm: tits conjecure}
Let~$A(\Gamma)$ be an Artin group. The subgroup generated by the squares of vertices~$\{v^2 \mid v \in V\}$ is a RAAG with presentation:
\[
\gp{t_v \text{ for } v \in V}{ t_v t_u = t_u t_v \text{ whenever } \{u, v\} \in E},
\]
where the generator~$t_v$ corresponds to~$v^2$. 
\end{theorem}
Note that if two vertices~$a, b$ commute, their squares~$a^2, b^2$ commute as well, so it is reasonable to have~$a^2 b^2 = b^2 a^2$ as a relation. On the other hand, if~$a, b$ are connected by an edge with label~$p > 2$, the result shows that the corresponding squares~$a^2, b^2$ are free of relations.

Now, using \autoref{thm: tits conjecure} on \autoref{ex: chordal Artin squares} one concludes that the subgroup~$\sgp{a^2, \, b^2, \, c^2, \, d^2}$ of the Artin group~$A(S_p)$ (the same holds for~$A(S_{p, q})$ as well) is isomorphic to~$A(C_4)$. By abuse of notation, we use a graphical presentation for this subgroup as below:
\begin{figure}[H]
\centering
\begin{tikzpicture}[scale = 0.8, >={Straight Barb[length=7pt,width=6pt]},thick]

\draw[fill=black] (2,1) circle (1pt) node[left] {$a^2$};
\draw[fill=black] (4,1) circle (1pt) node[right] {$b^2$};
\draw[fill=black] (4,-1) circle (1pt) node[right] {$c^2$};
\draw[fill=black] (2,-1) circle (1pt) node[left] {$d^2$};

\draw[] (1, 0) node[left] {$\sgp{a^2, \, b^2, \, c^2, \, d^2} \simeq~$};

\draw[thick] (2, 1) -- (4, 1);
\draw[thick] (4,1) -- (4,-1);
\draw[thick] (4,-1) -- (2,-1);
\draw[thick] (2,-1) -- (2,1);
\end{tikzpicture}
\end{figure}
This shows that both~$A(S_p)$ and~$A(S_{p, q})$ contain~$A(C_4)$ as a poisonous subgroup. 

This approach is not always useful, as explained in the next example.
\begin{example}\label{ex: Path Artin groups containing A(P_4)}
Let~$P_3^{m, n}$ be the Artin graph from category (3) appearing in \autoref{thm: separable subgroup Artin groups}, and let~$A_{m,n} = \langle \, a,\, b,\, c \mid (a, b)_m = (b, a)_m,\, (b, c)_n = (c, b)_n \, \rangle$
the Artin group based on~$P_3^{m, n}$. 

Obviously~$A_{2, 2} = \Z \times F_2 \in \mathcal{A}(Rat)$. Now suppose that~$m \leq n$ and that~$n$ is greater than~$2$.

Using Tits conjecture we see that the subgroup~$\sgp{a^2, b^2, c^2}$ is either isomorphic to~$\Z^2 \ast \Z$ (if~$m = 2$), or to~$F_3$ (if~$m > 2$), and both of them belong to~$\mathcal{A}(Rat)$; however, we cannot conclude that this is the case as well for~$A_{m,n}$ because this subgroup is not of finite index. In fact, Tits conjecture is not useful in this example; we will soon see that~$A_{m,n}$ contains~$A(P_4)$ whenever~$m + n > 4$, in which case~$A_{m,n} \notin \mathcal{A}(Rat)$.
\end{example}

To show~$A(P_4) \xhookrightarrow{} A(P_3^{m, n})$ for~$m + n > 4$, we use something stronger than Tits conjecture.

\subsubsection*{Generalized Tits Conjecture}
In \citep*{jankiewicz2022right} the authors are interested in finding as 'large' RAAGs inside Artin groups as possible. They have stated the Generalized Tits Conjecture, which yields for each Artin group a subgroup that is a RAAG. They do it by taking high powers of centers of irreducible spherical Artin subgroups, and conjecture that they generate a RAAG. They manage to prove the conjecture for some subclasses of Artin groups, two of which include the following:
\begin{itemize}
\item locally reducible Artin groups;
\item irreducible spherical Artin groups that are not of type~$E_6$,~$E_7$,~$E_8$. 
\end{itemize}
This is sufficient to generalize the Lohrey-Steinberg criterion and establish a classification of all Artin groups with decidable submonoid and rational subset membership problems.

We will not delve much into details, but we will motivate the approach with examples of dihedral and triangle Artin groups.
\begin{example} Let~$m > 2$. The Crisp-Paris RAAG for the dihedral Artin group
\begin{figure}[H]
\centering
\begin{tikzpicture}[>={Straight Barb[length=7pt,width=6pt]},thick]
\draw[] (-.5, 0) node[left] {$D_{m} = \gp{a, b}{(a, b)_m = (b, a)_m}$, based on~$\delta_m =~$};
\draw[fill=black] (0,0) circle (1.5pt) node[below] {$a$};
\draw[fill=black] (2,0) circle (1.5pt) node[below] {$b$};
\draw[thick] (0,0) -- (2,0);
\draw[fill=black] (1,0)  node[above] {$m$};
\end{tikzpicture}
\end{figure} 
is a free group of rank~$2$ generated by~$a^2$ and~$b^2$, and we can express it graphically as:
\begin{figure}[H]
\centering
\hspace{-2cm}
\begin{tikzpicture}[>={Straight Barb[length=7pt,width=6pt]},thick]
\draw[] (-.5, 0) node[left] {C -- P RAAG~$=$ };
\draw[fill=black] (0,0) circle (1.5pt) node[below] {$a^2$};
\draw[fill=black] (2,0) circle (1.5pt) node[below] {$b^2$};
\end{tikzpicture}
\end{figure} 
There is a way to do better, because we already know from \autoref{subsec: Dihedral Artin groups} that~$\Z \times F_2$ embeds in~$D_m$. Dihedral Artin groups are of finite type in particular, so they have infinite cyclic centers. Denote~$z_{ab} = (a, b)_m$. If~$m$ is even,~$z_{ab}$ is the generator of the center of~$D_m$; instead, if~$m$ is odd, then the center of~$D_m$ is generated by~$z_{ab}^2$.
\end{example}
\begin{remark}\label{rem: center commutes with squares}
The element~$z_{ab}^2$ commutes with~$a^2$ and~$b^2$ in~$D_m$ for any~$m \geq 2$.
\end{remark}
Since dihedral Artin groups are \emph{locally reducible Artin groups}, \citep[Theorem 1.1]{jankiewicz2022right} implies that the subgroup of~$D_m$ generated by~$a^2$,~$z_{ab}^2$, and~$b^2$ is a RAAG isomorphic to~$\Z \times F_2$, which can be expressed graphically as:
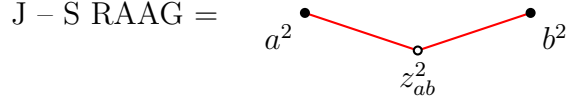
\begin{figure}[H]
\centering
\hspace{-2cm}
\begin{tikzpicture}[>={Straight Barb[length=7pt,width=6pt]},thick]
\draw[] (-1, 0) node[left] {J -- S RAAG~$=$ };

\draw[thick, red] (0,0) -- (1.5,-0.5);
\draw[thick, red] (3,0) -- (1.5,-0.5);
\draw[fill=black] (0,0) circle (1.5pt) node[below left] {$a^2$};
\draw[fill=black] (3,0) circle (1.5pt) node[below right] {$b^2$};
\draw[fill=white] (1.5,-0.5) circle (1.5pt) node[below] {$z_{ab}^2$};
\end{tikzpicture}
\end{figure} 
This will help to find~$A(P_4)$ inside~$A(P_3^{m,n})$ -- which is \emph{locally reducible}. Below we consider The Jankiewicz -- Schreve RAAG for paths of the form
\begin{figure}[H]
\centering
\begin{tikzpicture}[>={Straight Barb[length=7pt,width=6pt]},thick]
\draw[fill=black] (-1,0)  node[left] {$P_3^{m,n} =~$};
\draw[fill=black] (0,0) circle (1.5pt) node[below] {$a$};
\draw[fill=black] (2,0) circle (1.5pt) node[below] {$b$};
\draw[fill=black] (4,0) circle (1.5pt) node[below] {$c$};
\draw[thick] (0,0) -- (2,0);
\draw[thick] (2,0) -- (4,0);
\draw[fill=black] (1,0)  node[above] {$m$};
\draw[fill=black] (3,0)  node[above] {$n$};
\end{tikzpicture}
\end{figure} 
\begin{lemma}[{\citealp[Theorem 1.1]{jankiewicz2022right}}]\label{lem: P_3^{m,n} contains P_4}
Suppose~$m \leq n$ and~$m + n > 4$. Denote~$z_{ab} = (a, b)_m$ and~$z_{bc} = (b, c)_n$.
The subgroup~$\sgp{z_{ab}^2, b^2, z_{bc}^2, c^2}$ of~$A(P_3^{m,n})$ is a RAAG, isomorphic to~$A(P_4)$, and can be expressed graphically as:
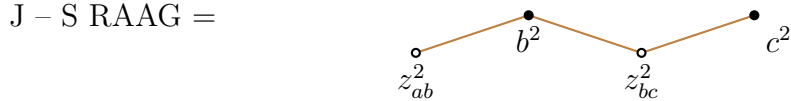
\begin{figure}[H]
\centering
\hspace{-2cm}
\begin{tikzpicture}[>={Straight Barb[length=7pt,width=6pt]},thick]
\draw[] (-1, 0) node[left] {J -- S RAAG~$=$ };

\draw[thick, brown] (3,0) -- (1.5,-0.5);
\draw[thick, brown] (3,0) -- (4.5,-0.5);
\draw[thick, brown] (6,0) -- (4.5,-0.5);
\draw[fill=black] (3,0) circle (1.5pt) node[below] {$b^2$};
\draw[fill=black] (6,0) circle (1.5pt) node[below right] {$c^2$};
\draw[fill=white] (1.5,-0.5) circle (1.5pt) node[below] {$z_{ab}^2$};
\draw[fill=white] (4.5,-0.5) circle (1.5pt) node[below] {$z_{bc}^2$};
\end{tikzpicture}
\end{figure} 
Note that if~$m > 2$ there is a larger RAAG inside, namely the one below:
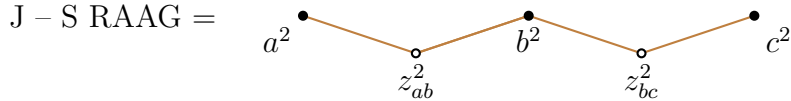
\begin{figure}[H]
\centering
\hspace{-2cm}
\begin{tikzpicture}[>={Straight Barb[length=7pt,width=6pt]},thick]
\draw[] (-1, 0) node[left] {J -- S RAAG~$=$ };

\draw[thick, brown] (0,0) -- (1.5,-0.5);
\draw[thick, brown] (3,0) -- (1.5,-0.5);
\draw[thick, brown] (3,0) -- (1.5,-0.5);
\draw[thick, brown] (3,0) -- (4.5,-0.5);
\draw[thick, brown] (6,0) -- (4.5,-0.5);
\draw[fill=black] (0,0) circle (1.5pt) node[below left] {$a^2$};
\draw[fill=black] (3,0) circle (1.5pt) node[below] {$b^2$};
\draw[fill=black] (6,0) circle (1.5pt) node[below right] {$c^2$};
\draw[fill=white] (1.5,-0.5) circle (1.5pt) node[below] {$z_{ab}^2$};
\draw[fill=white] (4.5,-0.5) circle (1.5pt) node[below] {$z_{bc}^2$};
\end{tikzpicture}
\end{figure} 
\end{lemma}
\begin{cor}
The Artin group~$A(P_3^{m,n})$ has decidable rational subset membership problem if and only if~$m = n = 2$, in which case~$A(P_3^{2, 2}) \simeq \Z \times F_2$. Otherwise it contains~$A(P_4)$.
\end{cor} 

Below we consider The Jankiewicz -- Schreve RAAG for triangles of the form~$\Delta_{l,m,n}$ from category (3) of \autoref{thm: separable subgroup Artin groups}. Every triangle with at most one unlabeled edge is either locally reducible, or irreducible spherical (and of course, not of type ~$E_6$,~$E_7$,~$E_8$), so the RAAG provided by the generalized Tits conjecture injects in~$A(\Delta_{l,m,n})$.

\begin{remark}\label{Triangle Artin groups containing A(P_4)}
Consider the triangle Artin graph:
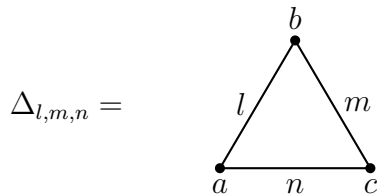
\begin{figure}[H]
\centering
\begin{tikzpicture}[>={Straight Barb[length=7pt,width=6pt]},thick]
\draw[] (-1, 0.8) node[left] {$\Delta_{l,m,n} =~$};
\draw[fill=black] (0,0) circle (1.5pt) node[below] {$a$};
\draw[fill=black] (2,0) circle (1.5pt) node[below] {$c$};
\draw[fill=black] (1,1.7) circle (1.5pt) node[above] {$b$};
\draw[thick] (0,0) -- node[below]{$n$}  (2,0);
\draw[thick] (2,0) -- node[right]{$m$}  (1,1.7);
\draw[thick] (0,0) -- node[left]{$l$} (1,1.7);
\end{tikzpicture}
\end{figure} 
and~$A_{l, m, n} = \langle \, a,\, b,\, c \mid (a, b)_l = (b, a)_l,\, (b, c)_m = (c, b)_m \, (c, a)_n = (a, c)_n \,\rangle$,
the Artin group based on~$\Delta_{l, m, n}$. 
Assume~$2  \leq l \leq m \leq n$.
If~$l = m = 2$, then~$A_{2,2,n} = \Z \times D_n$, which belongs to~$\mathcal{A}(Rat)$. Now suppose that both~$m, n$ are greater than~$2$. 
In this case~$A(P_4)$ embeds in~$A_{l, m, n}$ by~ \cite[Theorem 1.1 and Theorem 1.2]{jankiewicz2022right}. Indeed, the subgroup~$\sgp{b^2, z_{bc}^2, c^2, z_{ac}^2}$, is a RAAG isomorphic to~$A(P_4)$, because it is a subgroup of the Jankiewicz -- Schreve RAAG generated by its standard generators.
\end{remark}
\begin{cor}\label{cor: Delta_{l, m, n} contains P_4}
The Artin group~$A(\Delta_{l, m, n})$ has decidable submonoid membership problem if and only if at least two of the labels~$m, n, p$ are equal to~$2$. Otherwise it contains~$A(P_4)$.
\end{cor}

\begin{example} Below we provide an example of the triangle~$\Delta_{2, 3, 4}$ which is irreducible of spherical type. The graph, and the Crisp -- Paris RAAG graph are given below:
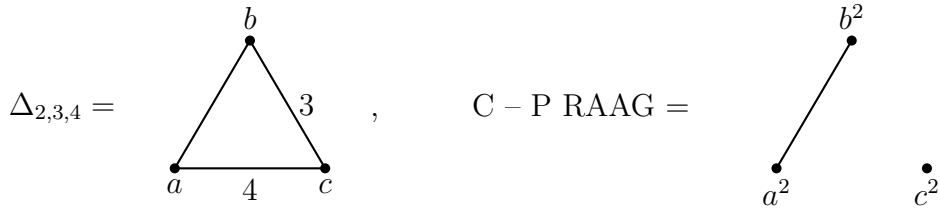
\begin{figure}[H]
\centering
\begin{tikzpicture}[>={Straight Barb[length=7pt,width=6pt]},thick]
\draw[] (-.5, 0.8) node[left] {$\Delta_{2,3,4} =~$};
\draw[fill=black] (0,0) circle (1.5pt) node[below] {$a$};
\draw[fill=black] (2,0) circle (1.5pt) node[below] {$c$};
\draw[fill=black] (1,1.7) circle (1.5pt) node[above] {$b$};
\draw[thick] (0,0) -- node[below]{$4$}  (2,0);
\draw[thick] (2,0) -- node[right]{$3$}  (1,1.7);
\draw[thick] (0,0) -- node[left]{} (1,1.7);

\draw[] (7, 0.8) node[left] {,\quad \quad \quad C -- P RAAG = };
\draw[fill=black] (8,0) circle (1.5pt) node[below] {$a^2$};
\draw[fill=black] (10,0) circle (1.5pt) node[below] {$c^2$};
\draw[fill=black] (9,1.7) circle (1.5pt) node[above] {$b^2$};

\draw[thick] (8,0) -- (9,1.7);
\end{tikzpicture}
\end{figure} 

The subgroups~$\sgp{b, c}$ and~$\sgp{a, c}$ of~$A(\Delta_{2,3,4})$ are dihedral and of spherical type in particular, so they have infinite cyclic centers. In spirit of \autoref{rem: center commutes with squares} we denote by~$z_{bc}^2$ and~$z_{a c}^2$ respectively the generators that will appear in the Jankiewicz -- Schreve RAAG. \\
Moreover,~$A(\Delta_{2,3,4})$ is of spherical type as well, so it has a non-trivial center: denote by~$z_{abc}^2$ the generator that will appear in the Jankiewicz -- Schreve RAAG, which we draw below.
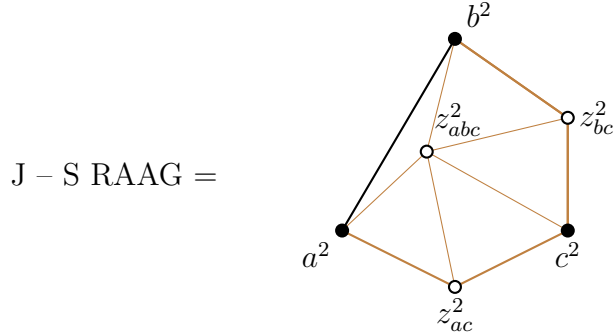
\begin{figure}[H]
\centering
\begin{tikzpicture}[scale = 1.5, >={Straight Barb[length=7pt,width=6pt]},thick]

\draw[] (-1, 0.5) node[left] {J -- S RAAG~$=$ };


\draw[thick] (0,0) -- (1,1.7);

\draw[thick, brown] (0,0) -- (1,-0.5);
\draw[thick, brown] (1,-0.5) -- (2,0);

\draw[thick, brown] (1, 1.7) -- (2,1);
\draw[thick, brown] (2,1) -- (2,0);

\draw[thin, brown] (0,0) -- (0.75,0.7);
\draw[thin, brown] (1, 1.7)  -- (0.75,0.7);
\draw[thin, brown] (2,0) -- (0.75,0.7);
\draw[thin, brown] (1,-0.5) -- (0.75,0.7);
\draw[thin, brown] (2, 1) -- (0.75,0.7);

\draw[thick, brown] (1, 1.7) -- (2,1);
\draw[thick, brown] (2,1) -- (2,0);

\draw[fill=black] (0,0) circle (1.5pt) node[below left] {$a^2$};
\draw[fill=black] (2,0) circle (1.5pt) node[below] {$c^2$};
\draw[fill=black] (1,1.7) circle (1.5pt) node[above right] {$b^2$};

\draw[fill=white] (1,-0.5) circle (1.5pt) node[below] {$z_{ac}^2$};

\draw[fill=white] (2,1) circle (1.5pt) node[right] {$z_{bc}^2$};

\draw[fill=white] (0.75,0.7) circle (1.5pt);

\draw[fill=white] (0.715,0.735) circle (0pt)
node[above right] {$z_{a bc}^2$};
\end{tikzpicture}
\end{figure} 
\end{example}

The following is the containment~$\mathcal{A} \setminus \mathcal{A}(P_4 \lor C_4) \subseteq \mathcal{A}(Rat)$, for which we provide a proof with a slightly new flavour. The idea is to identify labeled Artin graphs~$\Gamma$ that do not contain as induced subgraphs any of the graphs listed in \autoref{thm: separable subgroup Artin groups}. Then one uses the characterization of such graphs and closure properties of rational subset membership problem to deduce the result.
These closure properties are the same (in the context of this proof) as those of subgroup separability in the proof of \cite[Theorem 1.1]{almeida2024subgroupseparabilityartingroups}.
\begin{theorem}\label{thm: when we have decidable RSMP}
The Artin group~$A(\Gamma)$
has decidable rational subset membership problem if~$\Gamma$ does not contain as induced subgraphs any of the graphs listed in \autoref{thm: separable subgroup Artin groups}.
\end{theorem}

\begin{proof}
It is enough to show the result when~$\Gamma$ is connected, since decidability of rational subset membership problem is preserved under free products (see point 5. in \autoref{thm: class of groups with decidable rational subset membership}).

If~$\Gamma$ does not have any edge labeled by~$p > 2$, then~$\Gamma$ is a transitive forest (as it does not contain~$P_4$ and~$C_4$). In this case the result follows by \autoref{thm: main theorem for RAAGs}.

Now suppose that~$\Gamma$ contains an edge~$e = \{x, y\}$ with label~$p > 2$. If~$\Gamma$ does not contain other vertices then~$A(\Gamma) = D_p$, and~$D_p \in \mathcal{A}(Rat)$ by \autoref{cor: dihedral Artin groups are in A(Rat)}.

Now assume~$|V\Gamma| > 2$. As~$\Gamma$ is connected there is at least one vertex adjacent to~$x$ or~$y$; since~$\Gamma$ does not contain paths of the form~$P_3^{m,n}$ for~$m + n > 4$, any vertex of~$V\Gamma \setminus \{x, y\}$ which is adjacent to~$x$ or~$y$ should be adjacent to both; moreover both the edges~$\{a, x\}$ and~$\{a, y\}$ are unlabeled because the only triangles that~$\Gamma$ can contain have at least two unlabeled edges. Denote by~$L$ the vertices on the link of the edge~$e = \{x, y\}$, and let~$F = V\Gamma \setminus (L \cup \{x, y\})$ as in the figure below.
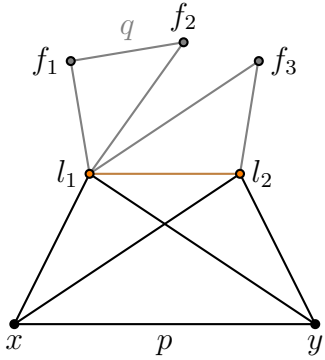
\begin{figure}[H]
\centering
\begin{tikzpicture}[>={Straight Barb[length=7pt,width=6pt]},thick]

\draw[thick] (0,0) -- node[below]{$p$} (4,0);

\draw[thick, gray] (1,2) --  (0.75, 3.5);
\draw[thick, gray] (3,2) --  (3.25, 3.55);

\draw[thick, gray] (0.75,3.5) -- node[above]{$q$} (2.25,3.75);

\draw[thick, gray] (1,2) -- (2.25,3.75);
\draw[thick, gray] (1,2) -- (3.25,3.5);

\draw[thick, brown] (1,2) --  (3, 2);

\draw[thick] (0,0) --  (1,2);
\draw[thick] (0,0) --  (3,2);

\draw[thick] (4,0) --  (1,2);
\draw[thick] (4,0) --  (3,2);

\draw[fill=black] (0,0) circle (1.5pt) node[below] {$x$};
\draw[fill=black] (4,0) circle (1.5pt) node[below] {$y$};

\draw[fill=orange] (1,2) circle (1.5pt) node[left] {$l_1$};
\draw[fill=orange] (3,2) circle (1.5pt) node[right] {$l_2$};


\draw[fill=gray] (0.75,3.5) circle (1.5pt) node[left] {$f_1$};
\draw[fill=gray] (2.25,3.75) circle (1.5pt) node[above] {$f_2$};
\draw[fill=gray] (3.25,3.5) circle (1.5pt) node[right] {$f_3$};
\end{tikzpicture}
\end{figure} 

{\bf Claim 1.} The subgraph induced by~$L$ is complete with unlabeled edges.

{\it Proof of Claim 1.} Consider these two cases:
\begin{itemize}
\item[(i)] If two vertices~$l_1, l_2 \in L$ are not connected by an edge, then the subgraph induced by the vertex set~$\{x, y, l_1, l_2\}$ would be isomorphic to the chordal Artin square~$S_p$, contradicting our assumption. 
\item[(ii)] On the other hand, if two vertices~$l_1, l_2 \in L$ are connected by an edge with label~$q > 2$, then the subgraph induced by the vertex set~$\{x, y, l_1, l_2\}$ would be isomorphic to the chordal Artin square~$S_{p,q}$, contradicting our assumption. 
\end{itemize}
This finishes proof of Claim 1. Note that if~$|F| = 0$, then~$A(\Gamma) = D_p \times \Z^{|L|}$ which would have decidable rational subset membership problem, because it is an iterated direct product of~$D_p$ by~$\Z$.
Note that if~$|F| = 0$, then every vertex~$l \in L$ is central -- connected to every vertex in~$V\Gamma \setminus \{l\}$ by unlabeled edges. Moreover, one has the following also when~$|F| > 0$.

{\bf Claim 2.} 
There is a vertex~$l \in L$ which is central in~$\Gamma$.

{\bf Observation.} Every edge with one endpoint in~$L$ and the other in~$F$ is unlabeled. Otherwise if~$\{l, f\} \in E\Gamma$ with~$l \in L$, and~$f \in F$ has a label~$k > 2$, then~$\Gamma$ would contain an induced subgraph (the one with vertex set~$\{x, l, f\}$) of the form~$P_3^{2, k}$ with~$k > 2$, which is not allowed. 

The proof of Claim 2 comes from the proof of the two results \cite[Lemma 4.5, Lemma 4.6]{almeida2021subgroup}. Indeed, their results imply that~$\Gamma$ has a central vertex, as~$|V\Gamma| \geq 3$. The fact that this central vertex belongs to~$L$ is obvious, because for any vertex~$u \in F \cup \{x, y\}$ there is a vertex~$u' \in V\Gamma$ such that~$v, v'$ are not adjacent. The other reason why we can use those results comes from~$\mathcal{A}(Rat) \subseteq \mathcal{A} \setminus \mathcal{A}(P_4 \lor C_4) = \mathcal{A}(SgSep)$.

In the example of the given graph, we note that~$l_1$ is central, hence we have 
\[
A(\Gamma) = \Z \times A(\Gamma \setminus \{l\}) = \Z \times \left( A(\Gamma_1) \ast A(\Gamma_2) \right),
\]
where~$\Gamma_1$ and~$\Gamma_2$ are the connected components of~$\Gamma \setminus \{l\}$ drawn below.
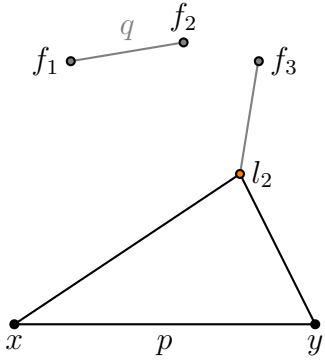
\begin{figure}[H]
\centering
\begin{tikzpicture}[>={Straight Barb[length=7pt,width=6pt]},thick]

\draw[thick] (0,0) -- node[below]{$p$} (4,0);

\draw[thick, gray] (3,2) --  (3.25, 3.55);

\draw[thick, gray] (0.75,3.5) -- node[above]{$q$} (2.25,3.75);



\draw[thick] (0,0) --  (3,2);

\draw[thick] (4,0) --  (3,2);

\draw[fill=black] (0,0) circle (1.5pt) node[below] {$x$};
\draw[fill=black] (4,0) circle (1.5pt) node[below] {$y$};

\draw[fill=orange] (3,2) circle (1.5pt) node[right] {$l_2$};


\draw[fill=gray] (0.75,3.5) circle (1.5pt) node[left] {$f_1$};
\draw[fill=gray] (2.25,3.75) circle (1.5pt) node[above] {$f_2$};
\draw[fill=gray] (3.25,3.5) circle (1.5pt) node[right] {$f_3$};
\end{tikzpicture}
\end{figure} 
Note that in general, by using Claim 2, one can express~$A(\Gamma)$ as a direct product of~$\Z$ and~$A(\Gamma \setminus \{l\})$, where the later is a free product of Artin groups based on graphs with less vertices than~$\Gamma$. Now the proof follows by induction on~$|V\Gamma|$.
\end{proof}

\bigskip
{\bf Final remark:} Note that Section 4 of \citep{gray2024membership} deals with the same problems as this article. Any overlapping of the presentation is not intentional.
\bigskip

\noindent{\textbf{{Acknowledgments}}} 
The author acknowledges support from the EPSRC Fellowship grant EP/V032003/1 ‘Algorithmic, topological and geometric aspects of infinite groups, monoids and inverse semigroups’.

\bibliography{main}

\begin{thebibliography}{19}
\providecommand{\natexlab}[1]{#1}
\providecommand{\url}[1]{\texttt{#1}}
\expandafter\ifx\csname urlstyle\endcsname\relax
  \providecommand{\doi}[1]{doi: #1}\else
  \providecommand{\doi}{doi: \begingroup \urlstyle{rm}\Url}\fi

\bibitem[Almeida and Lima(2021)]{almeida2021subgroup}
K.~Almeida and I.~Lima.
\newblock Subgroup separability of {Artin} groups.
\newblock \emph{Journal of Algebra}, 583:\penalty0 25--37, 2021.

\bibitem[Almeida and Lima(2024)]{almeida2024subgroupseparabilityartingroups}
K.~Almeida and I.~Lima.
\newblock Subgroup separability of {Artin} groups {II}, 2024.
\newblock URL \url{https://arxiv.org/abs/2403.05483}.

\bibitem[Antol{\'\i}n and Foniqi(2024)]{antolin2024subgroups}
Y.~Antol{\'\i}n and I.~Foniqi.
\newblock Subgroups of even {Artin} groups of {FC} type.
\newblock \emph{Journal of Group Theory}, 2024.
\newblock \doi{doi:10.1515/jgth-2023-0093}.
\newblock URL \url{https://doi.org/10.1515/jgth-2023-0093}.

\bibitem[Blasco-Garc{\'\i}a et~al.(2018)Blasco-Garc{\'\i}a,
  Mart{\'\i}nez-P{\'e}rez, and Paris]{blasco2018poly}
R.~Blasco-Garc{\'\i}a, C.~Mart{\'\i}nez-P{\'e}rez, and L.~Paris.
\newblock Poly-freeness of even {Artin} groups of {FC} type.
\newblock \emph{Groups, Geometry, and Dynamics}, 13\penalty0 (1):\penalty0
  309--325, 2018.

\bibitem[Bodart(2024)]{bodart2024membership}
C.~Bodart.
\newblock Membership problems in nilpotent groups.
\newblock \emph{arXiv preprint arXiv:2401.15504}, 2024.

\bibitem[Charney(2007)]{charney2007introduction}
R.~Charney.
\newblock An introduction to right-angled {Artin} groups.
\newblock \emph{Geometriae Dedicata}, 125\penalty0 (1):\penalty0 141–158,
  2007.

\bibitem[Crisp and Paris(2001)]{crisp2001solution}
J.~Crisp and L.~Paris.
\newblock The solution to a conjecture of {Tits} on the subgroup generated by
  the squares of the generators of an {Artin} group.
\newblock \emph{Inventiones mathematicae}, 145:\penalty0 19--36, 2001.

\bibitem[Dasbach et~al.(2011)Dasbach, Mangum, and
  Birman]{dasbach2001automorphism}
O.~T. Dasbach, B.~S. Mangum, and T.~J. Birman.
\newblock The automorphism group of a free group is not subgroup separable.
\newblock \emph{AMS IP studies in advanced mathematics}, 24:\penalty0 23--28,
  2011.

\bibitem[Droms(1987)]{droms1987subgroups}
C.~Droms.
\newblock Subgroups of graph groups.
\newblock \emph{J. ALGEBRA.}, 110\penalty0 (2):\penalty0 519--522, 1987.

\bibitem[Gray and Nyberg-Brodda(2024)]{gray2024membership}
R.~D. Gray and C.-F. Nyberg-Brodda.
\newblock Membership problems in braid groups and {Artin} groups.
\newblock \emph{arXiv preprint arXiv:2409.11335}, 2024.

\bibitem[Jankiewicz and Schreve(2022)]{jankiewicz2022right}
K.~Jankiewicz and K.~Schreve.
\newblock Right-angled {Artin} subgroups of {Artin} groups.
\newblock \emph{Journal of the London Mathematical Society}, 106\penalty0
  (2):\penalty0 818--854, 2022.

\bibitem[Kambites(2009)]{kambites2009commuting}
M.~Kambites.
\newblock On commuting elements and embeddings of graph groups and monoids.
\newblock \emph{Proceedings of the Edinburgh Mathematical Society}, 52\penalty0
  (1):\penalty0 155--170, 2009.

\bibitem[Kapovich et~al.(2005)Kapovich, Weidmann, and
  Myasnikov]{kapovich2005foldings}
I.~Kapovich, R.~Weidmann, and A.~Myasnikov.
\newblock Foldings, graphs of groups and the membership problem.
\newblock \emph{International Journal of Algebra and Computation}, 15\penalty0
  (01):\penalty0 95--128, 2005.

\bibitem[Katayama(2017)]{katayama2017raags}
T.~Katayama.
\newblock {RAAG}s in knot groups.
\newblock \emph{Geometry and Analysis of Discrete Groups and Hyperbolic
  Spaces}, 66:\penalty0 37--56, 2017.

\bibitem[Kim and Koberda(2013)]{kim2013embedability}
S.-h. Kim and T.~Koberda.
\newblock Embedability between right-angled {Artin} groups.
\newblock \emph{Geometry \& Topology}, 17\penalty0 (1):\penalty0 493--530,
  2013.

\bibitem[Lohrey and Steinberg(2008)]{lohrey2008submonoid}
M.~Lohrey and B.~Steinberg.
\newblock The submonoid and rational subset membership problems for graph
  groups.
\newblock \emph{Journal of Algebra}, 320\penalty0 (2):\penalty0 728--755, 2008.

\bibitem[Metaftsis and Raptis(2008)]{metaftsis2008profinite}
V.~Metaftsis and E.~Raptis.
\newblock On the profinite topology of right-angled {Artin} groups.
\newblock \emph{Journal of Algebra}, 320\penalty0 (3):\penalty0 1174--1181,
  2008.

\bibitem[Mikhailova(1966)]{mikhailova1966occurrence}
K.~Mikhailova.
\newblock The occurrence problem for direct products of groups.
\newblock \emph{Matematicheskii Sbornik}, 112\penalty0 (2):\penalty0 241--251,
  1966.

\bibitem[{Van der Lek}(1983)]{van1983homotopy}
H.~{Van der Lek}.
\newblock \emph{The homotopy type of complex hyperplane complements}.
\newblock PhD thesis, Katholieke Universiteit te Nijmegen, 1983.

\end{thebibliography}

\noindent\textit{\\ Islam Foniqi,\\
The University of East Anglia\\ 
Norwich (United Kingdom)\\}
{email: i.foniqi@uea.ac.uk}

\end{document}